\documentclass[12pt,a4paper]{article}
\usepackage{mathrsfs}
\usepackage{indentfirst}
\usepackage{amsmath,amssymb,amsfonts,amsthm,graphics}
\usepackage{makeidx}
\usepackage{color}
\include{PDF}
\newtheorem{theorem}{Theorem}

\newtheorem{lemma}[theorem]{Lemma}

\newtheorem{corollary}[theorem]{Corollary}

\title{\bf The $(k,\ell)$-rainbow index of random graphs\footnote{Supported by NSFC and the ``973'' program.}}
\author{\small Qingqiong Cai, Xueliang Li, Jiangli Song\\
{\small Center for Combinatorics and LPMC-TJKLC}\\ {\small Nankai
University}\\ {\small Tianjin 300071, China}\\ {\small Email:
cqqnjnu620@163.com, lxl@nankai.edu.cn,
songjiangli@mail.nankai.edu.cn}}
\date{}

\begin{document}
\maketitle

\begin{abstract}
A tree in an edge colored graph is said to be a rainbow tree if no
two edges on the tree share the same color. Given two positive
integers $k$, $\ell$ with $k\geq 3$, the \emph{$(k,\ell)$-rainbow
index} $rx_{k,\ell}(G)$ of $G$ is the minimum number of colors
needed in an edge-coloring of $G$ such that for any set $S$ of $k$
vertices of $G$, there exist $\ell$ internally disjoint rainbow
trees connecting $S$. This concept was introduced by Chartrand et.
al., and there have been very few related results about it. In this
paper, We establish a sharp threshold function for
$rx_{k,\ell}(G_{n,p})\leq k$ and $rx_{k,\ell}(G_{n,M})\leq k,$
respectively, where $G_{n,p}$ and $G_{n,M}$ are the usually defined
random graphs.
\end{abstract}



\section {\large Introduction}

All graphs in this paper are undirected, finite and simple. We
follow \cite{Bondy} for graph theoretical notation and terminology
not described here. Let $G$ be a nontrivial connected graph with an
\emph{edge-coloring} $c: E(G)\rightarrow\{1, 2,\cdots, t\}, t \in
\mathbb{N}$, where adjacent edges may be colored the same. A path is
said to be a \emph{rainbow path} if no two edges on the path have
the same color. An edge-colored graph $G$ is called \emph{rainbow
connected} if for every pair of distinct vertices of $G$ there
exists a rainbow path connecting them. The \emph{rainbow connection
number} of a graph $G$, denoted by $rc(G)$, is defined as the
minimum number of colors that are needed in order to make $G$
rainbow connected. The \emph{rainbow $k$-connectivity} of $G$,
denoted by $rc_{k}(G)$, is defined as the minimum number of colors
in an edge-coloring of $G$ such that every two distinct vertices of
$G$ are connected by $k$ internally disjoint rainbow paths. These
concepts were introduced by Chartrand et. al. in \cite{ChartrandGP,
ChartrandGL}. Recently, there have been published a lot of results
on the rainbow connections. We refer the readers to \cite{LiSun3,
LiSun} for details.

Similarly, a tree $T$ is called a \emph{rainbow tree} if no
two edges of $T$ have the same color. For $S\subseteq V(G)$, a
\emph{rainbow $S$-tree} is a rainbow tree connecting the vertices of
$S$. Suppose that $\{T_{1},T_{2},\cdots, T_{\ell}\}$ is a set of
rainbow $S$-trees. They are called \emph{internally disjoint} if
$E(T_{i})\cap E(T_{j})=\emptyset$ and $V(T_{i})\bigcap V(T_{j})=S$
for every pair of distinct integers $i,j$ with $1\leq i,j\leq \ell$
(Note that the trees are vertex-disjoint in $G\setminus S$). Given
two positive integers $k$, $\ell$ with $k\geq 2$, the
\emph{$(k,\ell)$-rainbow index} $rx_{k,\ell}(G)$ of $G$ is the
minimum number of colors needed in an edge-coloring of $G$ such that
for any set $S$ of $k$ vertices of $G$, there exist $\ell$
internally disjoint rainbow $S$-trees. In particular, for $\ell=1$,
we often write $rx_{k}(G)$ rather than $rx_{k,1}(G)$ and call it the
\emph{$k$-rainbow index}. It is easy to see that
$rx_{2,\ell}(G)=rc_{\ell}(G)$. So the $(k,\ell)$-rainbow index can
be viewed as a generalization of the rainbow connectivity. In the
sequel, we always assume $k\geq 3$.

The concept of $(k,\ell)$-rainbow index was also introduced by
Chartrand et. al. in \cite{Chartrand}. They determined the
$k$-rainbow index of all unicyclic graphs and the $(3,\ell)$-rainbow
index of complete graphs for $\ell=1,2$. In \cite{Cai}, we
investigated the $(k,\ell)$-rainbow index of complete graphs for
every pair $k,\ell$ of integers. We proved that for every pair
$k,\ell$ of positive integers with $k\geq 3$, there exists a
positive integer $N=N(k,\ell)$ such that $rx_{k,\ell}(K_{n})=k$ for
every integer $n\geq N$.

In this paper, we study the $(k,\ell)$-rainbow index of random
graphs $G_{n,p}$ and $G_{n,M}$ and establish a sharp threshold
function for the property $rx_{k,\ell}(G_{n,p})$ $\leq k$ and
$rx_{k,\ell}(G_{n,M})$ $\leq k,$ respectively, where $G_{n,p}$ and
$G_{n,M}$ are defined as usual in \cite{BB}.

\section{Basic Concepts}

The two most frequently occurring probability models of random
graphs are $\mathcal{G}(n, p)$ and $\mathcal{G}(n, M)$. The model
$\mathcal{G}(n, p)$ consists of all graphs on $n$ vertices, in which
the edges are chosen independently and randomly with probability
$p$. Whereas the model $\mathcal{G}(n, M)$ consists of all graphs on
$n$ vertices and $M$ edges, in which each graph has the same
probability. Let $G_{n,p}$, $G_{n,M}$ stand for random graphs from
the models $\mathcal{G}(n, p)$ and $\mathcal{G}(n, M)$. We say that
an event $E=E(n)$ happens \emph{almost surely} (or a.s. for short)
if $\lim_{n\rightarrow\infty}Pr[E(n)]=1$. Let $F,G,H$ be graphs on
$n$ vertices. A property $Q$ is said to be \emph{monotone} if
whenever $G\subseteq H$ and $G$ satisfies $Q$, then $H$ also
satisfies $Q$. Moreover, We call a property $Q$ \emph{convex} if
whenever $F\subset G \subset H$, $F$ satisfies $Q$ and $H$ satisfies
$Q$, then $G$ also satisfies $Q$. For a graph property $Q$, a
function $p(n)$ is called a \emph{threshold function} of $Q$ if

$\bullet$  $\frac{p^{'}(n)}{p(n)} \rightarrow 0$, then $G_{n,
p^{'}(n)}$ almost surely does not satisfy $Q$; and

$\bullet$ $\frac{p^{''}(n)}{p(n)} \rightarrow \infty$, then $G_{n,
p^{''}(n)}$ almost surely satisfies $Q$.

Furthermore, $p(n)$ is called a \emph{sharp threshold function} of
$Q$ if there are two positive constants $c$ and $C$ such that

$\bullet$ for every $p^{'}(n) \leq cp(n)$, $G_{n, p^{'}(n)}$ almost
surely does not satisfy $Q$; and

$\bullet$ for every $p^{''}(n) \geq Cp(n)$, $G_{n, p^{''}(n)}$
almost surely satisfies $Q$.

Similarly, we can define $M(n)$ as a threshold function of $Q$ in
the model $\mathcal{G}(n, M)$.

It is well known that all monotone graph properties have a threshold
function \cite{BB}. Obviously, for every pair $k, \ell$ of positive
integers, the property that the $(k,\ell)$-rainbow index is at most
$k$ is monotone, and thus has a threshold.

\section{Main results}

In this section, we study the $(k,\ell)$-rainbow index of random graphs.
\begin{theorem}\label{thm4.1}
For every pair $k,\ell$ of positive integers with $k\geq3$,
$\sqrt[^k\!]{\frac{log_{a}n}{n}}$ is a sharp threshold function for
the property $rx_{k,\ell}(G_{n,p})\leq k$, where
$a=\frac{k^{k}}{k^{k}-k!}$.
\end{theorem}
\begin{proof}
The proof will be two-fold. For the first part, we show that, there
exists a positive constant $c_{1}$ such that for every $p\geq
c_{1}\sqrt[^k\!]{\frac{log_{a}n}{n}}$, almost surely
$rx_{k,\ell}(G_{n,p})\leq k$, which can be derived from the
following two claims.

\emph{\textbf{Claim 1:}} For any $c_{1}\geq3$, if $p\geq
c_{1}\sqrt[^k\!]{\frac{log_{a}n}{n}}$, then almost surely any $k$
vertices in $G_{n,p}$ have at least $2klog_{a}n$ common neighbors.

For any $S\in V(G_{n,p})$ with $|S|=k$, let $D(S)$ denote the event
that the vertices in $S$ have at least $2klog_{a}n$ common
neighbors. Then it suffices to prove that, for $p=
c_{1}\sqrt[^k\!]{\frac{log_{a}n}{n}}$, $Pr[\ \bigcap\limits_{S}D(S)\
]\rightarrow 1$, as $n\rightarrow \infty$. Define $Y$ as the number
of common neighbors of all the vertices in $S$. Then $Y\thicksim
Bin(n-k, (c_{1}\sqrt[^k\!]{\frac{log_{a}n}{n}})^{k})$ and
$E(Y)=c_{1}^{k}\frac{n-k}{n}log_{a}n$. Assume that
$n>\frac{c_{1}^{k}k}{c_{1}^{k}-2k}$. Using the Chernoff Bound
\cite{Alon}, we get that

\begin{eqnarray*}
 Pr[\overline{D(S)}] &=& Pr[Y<2klog_{a}n] \\
  &=& Pr[Y<\frac{c_{1}^{k}(n-k)}{n}log_{a}n(1-\frac{(c_{1}^{k}-2k)n-c_{1}^{k}k}{c_{1}^{k}(n-k)})]\\
  &\leq&  e^{-\frac{c_{1}^{k}(n-k)}{2n}log_{a}n(\frac{(c_{1}^{k}-2k)n-c_{1}^{k}k}{c_{1}^{k}(n-k)})^{2}}\\
  &<& n^{-\frac{c_{1}^{k}(n-k)}{2n}(\frac{(c_{1}^{k}-2k)n-c_{1}^{k}k}{c_{1}^{k}(n-k)})^{2}}.\\
\end{eqnarray*}

Note that the assumption $n>\frac{c_{1}^{k}k}{c_{1}^{k}-2k}$ ensures
$\frac{(c_{1}^{k}-2k)n-c_{1}^{k}k}{c_{1}^{k}(n-k)}>0$. So we can
apply the Chernoff Bound to scaling the above inequalities. The last
inequality holds, since $1<a=\frac{k^{k}}{k^{k}-k!}<e$ and then
$log_{a}n>\ln n$.

It follows from the union bound that
\begin{eqnarray*}
  Pr[\ \bigcap\limits_{S} D_{S}\ ] &=& 1- Pr[\ \bigcup\limits_{S} \overline{D_{S}}\ ] \\
  &\geq& 1-\sum\limits_{S}Pr[\ \overline{D_{S}}\ ]\\
  &>& 1-{n \choose k}n^{-\frac{c_{1}^{k}(n-k)}{2n}(\frac{(c_{1}^{k}-2k)n-c_{1}^{k}k}{c_{1}^{k}(n-k)})^{2}} \\
   &>& 1-n^{k-\frac{c_{1}^{k}(n-k)}{2n}(\frac{(c_{1}^{k}-2k)n-c_{1}^{k}k}{c_{1}^{k}(n-k)})^{2}}.\\
\end{eqnarray*}
It is not hard to see that $c_{1}>3$ can guarantee
$k-\frac{c_{1}^{k}(n-k)}{2n}(\frac{(c_{1}^{k}-2k)n-c_{1}^{k}k}{c_{1}^{k}(n-k)})^{2}<0$
for sufficiently large $n$. Then $\lim\limits_{n\rightarrow\infty}
1-n^{k-\frac{c_{1}^{k}(n-k)}{2n}(\frac{(c_{1}^{k}-2k)n-c_{1}^{k}k}{c_{1}^{k}(n-k)})^{2}}=1$,
which implies $\lim\limits_{n\rightarrow\infty}Pr[\
\bigcap\limits_{S} D_{S}\ ]=1$ as desired.

\emph{\textbf{Claim 2:}} If any $k$ vertices in $G(n,p)$ have at
least $2klog_{a}n$ common neighbors, then there exists a positive
integer $N=N(k)$ such that $rx_{k,\ell}(G(n,p))\leq k$ for every
integer $n\geq N$.

Let $C=\{1,2,\cdots,k\}$ be a set of $k$ different colors. We color
the edges of $G(n,p)$ with the colors from $C$ randomly and
independently. For $S\subseteq V(G(n,p))$ with $|S|=k$, define
$E(S)$ as the event that there exist at least $\ell$ internally
disjoint rainbow $S$-trees. It suffices to prove that Pr[ $\bigcap\limits_{S}E(S)$ ]$>0$.

Suppose $S=\{v_{1}, v_{2}, \cdots, v_{k}\}$. For any common neighbor
$u$ of the vertices in $S$, let $T(u)$ denote the star with
$V(T(u))=\{u,v_{1}, v_{2},\cdots,$ $v_{k}\}$ and $E(T(u))=\{uv_{1},
uv_{2}, \cdots, uv_{k}\}$. Set $\mathcal{T}=\{T(u)| u$ is a common
neighbor of the vertices in $S \}$. Then $\mathcal{T}$ is a set of
at least $2klog_{a}n$ internally disjoint $S$-trees. It is easy to
see that $q$:=Pr[T$\in \mathcal{T}$ is a rainbow
tree]=$\frac{k!}{k^{k}}<\frac{1}{2}$. So $1-q>q$. Define $X$ as the
number of rainbow S-trees in $\mathcal{T}$. Then we have
\begin{eqnarray*}
 Pr[\overline{E(S)}] &\leq& Pr[X\leq\ell-1] \\
  &\leq& \sum\limits_{i=0}^{\ell-1}{2klog_{a}n \choose i}q^{i}(1-q)^{2klog_{a}n-i}\\
  &\leq& (1-q)^{2klog_{a}n}\sum\limits_{i=0}^{\ell-1}{2klog_{a}n \choose i}\\
  &\leq& (1-q)^{2klog_{a}n}(1+2klog_{a}n)^{\ell-1}\\
  &=& \frac{(1+2klog_{a}n)^{\ell-1}}{n^{2k}}.\\
\end{eqnarray*}
It yields that
\begin{eqnarray*}
  Pr[\ \bigcap\limits_{S} E(S)\ ] &=& 1- Pr[\ \bigcup\limits_{S} \overline{E(S)}\ ] \\
  &\geq& 1-\sum\limits_{S}Pr[\ \overline{E(S)}\ ]\\
  &\geq& 1-{n \choose k}\frac{(1+2klog_{a}n)^{\ell-1}}{n^{2k}} \\
   &>& 1-\frac{(1+2klog_{a}n)^{\ell-1}}{n^{k}}.\\
\end{eqnarray*}
Obviously,
$\lim\limits_{n\rightarrow\infty}1-\frac{(1+2klog_{a}n)^{\ell-1}}{n^{k}}=1$,
and then $\lim\limits_{n\rightarrow\infty}Pr[\ \bigcap\limits_{S}
E_{S}\ ]=1$. Thus there exists a positive integer $N=N(k)$ such that
$Pr[\ \bigcap\limits_{S} E_{S}\ ]>0$ for every integer $n\geq N$.

For the other direction, we show that there exists a positive
constant $c_{2}$ such that for every $p\leq
c_{2}\sqrt[^k\!]{\frac{log_{a}n}{n}}$, almost surely
$rx_{k,\ell}(G_{n,p})\geq k+1$.

It suffices to prove that for a sufficiently small constant $c_{2}$,
the random graph $G(n,p)$ with
$p=c_{2}\sqrt[^k\!]{\frac{log_{a}n}{n}}$ almost surely contains a
set $S$ of $k$ vertices satisfying

(i) $S$ is an independent set;

(ii) the vertices in $S$ have no common neighbors.

Clearly, for such $S$, there exists no rainbow $S$-trees in any
$k$-edge-coloring, which implies $rx_{k,\ell}(G_{n,p})\geq k+1$.

Fix a set $H$ of $n^{1/(2k+1)}$ vertices in $G_{n,p}$ (we may and
will assume that $n^{1/(2k+1)}/k$ is an integer). Let $E_{1}$ be the
event that $H$ is an independent set. Then
\begin{center}
pr[$E_{1}$]$=(1-c_{2}\sqrt[^k\!]{\frac{log_{a}n}{n}})^{n^{1/(2k+1)}
\choose 2}=1-o(1),$
\end{center}
where $o(1)$ denotes a function tending to 0 as n tends to infinity.

Partition $H$ into $t$ subsets $H_{1},H_{2}\ldots,H_{t}$
arbitrarily, where $t=n^{1/(2k+1)}/k$ and
$|H_{1}|=|H_{2}|=\ldots=|H_{t}|=k$. Let $E_{2}$ be the event that
there exists some $H_{i}$ without common neighbors in
$V(G_{n,p})\backslash H$. Then for sufficiently small $c_{2}$,
\begin{center}
Pr[$E_{2}$]$=1-(1-(1-c_{2}^{k}\frac{log_{a}n}{n})^{n-n^{1/(2k+1)}})^{n^{1/(2k+1)}/k}=1-o(1).$
\end{center}
So almost surely there exists some set $H_{i}$ of $k$ vertices
satisfying properties (i) and (ii). Thus for sufficiently small
$c_{2}$ and every $p\leq c_{2}\sqrt[^k\!]{\frac{log_{a}n}{n}}$,
almost surely $rx_{k,\ell}(G_{n,p})\geq k+1$. The proof is complete.
\end{proof}

Next we will turn to another well-known random graph model
$\mathcal{G}(n,M)$. We start with a useful lemma which reveals the
relationship between $\mathcal{G}(n,p)$ and $\mathcal{G}(n,M)$. Set
$N={n\choose2}$.
\begin{lemma}\cite{BB}\label{lem4}
 If $Q$ is a convex property and $p(1-p)N\rightarrow \infty$, then
 $G_{n,p}$ almost surely has $Q$ if and only if for every fixed $x$,
 $G_{n,M}$ almost surely has $Q$, where $M=\lfloor pN+x(p(1-p)N)^{1/2}\rfloor$.
\end{lemma}

Clearly, the property that the $(k,\ell)$-rainbow index of a given
graph is at most $k$, is a convex property. By Theorem \ref{thm4.1}
and Lemma \ref{lem4}, we get the following result:
\begin{corollary}
For every pair $k,\ell$ of positive integers with $k\geq3$,
$M(n)=\sqrt[^k\!]{n^{2k-1}log_{a}n}$ is a sharp threshold function
for the property $rx_{k,\ell}(G_{n,M})\leq k$, where
$a=\frac{k^{k}}{k^{k}-k!}$.
\end{corollary}

\noindent \textbf{Remark:} If $p$ is a threshold function for a
given property $Q$, then so is $\lambda p$ for any positive constant
$\lambda$. It follows that $p(n)=\sqrt[^k\!]{\frac{logn}{n}}$ (
$M(n)=\sqrt[^k\!]{n^{2k-1}logn}$ ) is also a sharp threshold
function for the property $rx_{k,\ell}(G_{n,p})\leq k$ (
$rx_{k,\ell}(G_{n,M})\leq k$ ), which correspond to the results in
\cite{Fujita}.

\end{document}